\documentclass{amsart}

\usepackage{amsmath, amsfonts, amssymb, amscd, enumerate}
\usepackage{color}
\usepackage[all]{xy}

\newtheorem{theo}{Theorem}[section]

\newtheorem{prop}[theo]{Proposition}

\newtheorem{example}[theo]{Example}
\newtheorem{definition}[theo]{Definition}

\newtheorem{lemma}[theo]{Lemma}
\newtheorem{theorem}[theo]{Theorem}

\newtheorem{proposition}[theo]{Proposition}

\newtheorem{remark}[theo]{Remark}

\renewcommand{\=}{:=}

\newcommand{\rank}{\operatorname{rank}}
\newcommand{\beq}{\begin{equation}}
\newcommand{\eeq}{\end{equation}}

\newcommand{\s}{\sigma}

\newcommand{\C}{\mathbb{C}}

\newcommand{\R}{\mathbb{R}}
\newcommand{\Z}{\mathbb{Z}}
\renewcommand{\P}{\mathbb{P}}
\renewcommand{\H}{\mathbb{H}}

\newcommand{\HH}{\mathbb{H}}
\newcommand{\HP}{\mathbb{HP}}

\newcommand{\rr}{{\mathbb{R}}}

\newcommand{{\ee}}{{\`e}}
\newcommand{{\aA }}{{\`a}}
\newcommand{{\oo}}{{\`o}}
\newcommand{{\uu}}{{\`u}}
\newcommand{{\ii}}{{\`i}}
\newcommand{{\BL}}{{Bl_0(\HH^n)}}
\newcommand{{\BLHP}}{{Bl_{[1,0,0]}(\HP^2)}}

\newcommand\SO{\mathrm{SO}}

\newcommand\T{\mathrm{T}}

\newcommand\Sp{\mathrm{Sp}}

\renewcommand{\square}{\kern1pt\vbox
{\hrule height 0.6pt\hbox{\vrule width 0.6pt\hskip 3pt
\vbox{\vskip 6pt}\hskip 3pt\vrule width 0.6pt}\hrule height0.6pt}\kern1pt}

\renewcommand\Im{\operatorname{Im}}

\renewcommand{\Im}{{\rm Im}}

\def\<#1,#2>{\langle\,#1,\,#2\,\rangle}
\newcommand{\arr}{\begin{array}{rlll}}
\newcommand{\ea}{\end{array}}
\newcommand{\bea}{\begin{eqnarray}}
\newcommand{\eea}{\end{eqnarray}}
\newcommand{\bean}{\begin{eqnarray*}}
\newcommand{\eean}{\end{eqnarray*}}

\def\sideremark#1{\ifvmode\leavevmode\fi\vadjust{
\vbox to0pt{\hbox to 0pt{\hskip\hsize\hskip1em
\vbox{\hsize3cm\tiny\raggedright\pretolerance10000
\noindent #1\hfill}\hss}\vbox to8pt{\vfil}\vss}}}

\newcounter{ssig}
\newcounter{ttig}
\begin{document}
\author{Graziano Gentili}
\address{Dipartimento di Matematica e Informatica ``U. Dini'', Universit\`a di Firenze, 50134 Firenze, Italy}

\email{gentili@math.unifi.it}

\author{Anna Gori}
\address{Dipartimento di Matematica, Universit\`a di Milano, Via Saldini 50, 20133 Milano, Italy}

\email{anna.gori@unimi.it}

\author{Giulia Sarfatti}
\address{Dipartimento di Matematica e Informatica ``U. Dini'', Universit\`a di Firenze, 50134 Firenze, Italy}
\email{sarfatti@math.unifi.it}
\thanks{\rm This project has been supported by G.N.S.A.G.A. of INdAM - Rome (Italy), by MIUR of the Italian Government (Research Projects: PRIN ``Real and complex manifolds: geometry, topology and harmonic analysis'' and FIRB ``Geometric function theory and differential geometry"). The third author is partially supported by SIR ``Analytic aspects in complex and hypercomplex geometry" of the Italian MIUR and part of this project has been developed while she was an INdAM COFUND fellow at IMJ-PRG  Paris.}

\

\title{Quaternionic toric manifolds}

\title{Quaternionic toric manifolds}
\begin{abstract}  In the present paper we  introduce and study  a  new notion of toric manifold in the quaternionic setting. We develop  a  construction with which, starting from appropriate  $m$-dimensional Delzant polytopes, we obtain manifolds of real dimension $4m$, acted on  by $m$ copies of the group  $\Sp(1)$ of unit quaternions. These manifolds, are {\em quaternionic regular} in the sense of \cite{GGS} and can be endowed  with a $4$-plectic structure and  a generalized moment map.  Convexity properties of the image of the moment map are studied.\\
Quaternionic toric manifolds appear to be a large enough class of examples where one can test and study  new results in quaternionic geometry.

\end{abstract}
\maketitle
\section{introduction}

 \noindent Toric varieties are geometric objects that can be defined by combinatorial information encoded in convex polyhedra. They  provide a large and interesting class of examples in algebraic geometry and many notions in this field  such as singularities, birational maps, cycles, homology, intersection theory can be interpreted in terms of properties of the convex polyedra on which these varieties are modeled.
An exahustive introduction to this topic can be found in the book \cite{cox} by Cox, Little and Schenck.

The study of toric manifolds (i.e., smooth toric varieties) has many different  motivations and a wide spectrum of applications. In particular, in symplectic geometry, toric manifolds provide examples of extremely symmetric and completely integrable Hamiltonian spaces.  Properties of symplectic toric manifolds are extensively studied, and are in the mainstream of current mathematical research.

The term moment map was introduced by Souriau, \cite{So}, under the French name {\em application moment}, to indicate one of the  main tools used  to study problems in geometry and topology when there is a suitable symmetry, as illustrated in the book by Gelfand, Kapranov and Zelevinsky \cite{Ge}.
The role of the moment map is fundamental in the symplectic setting: in fact the geometry encoded in its image, the so-called moment polytope, identifies the symplectic toric manifold.

Recent developments in the theory of regular functions over the quaternions encourage to go through the already existing approaches and to elaborate new tools to study quaternionic toric manifolds. This class of manifolds seems suitable to become large enough to give interesting examples, where to test and develop new results of quaternionic differential (and 4-plectic) geometry.\\
\\
\noindent The purpose of the present paper is to introduce the notion of quaternionic toric manifolds. 
The  starting point is the definition of $4$-plectic manifold, originally introduced by Foth in \cite{F}, as a natural generalization of symplectic manifold. In the symplectic case, whenever a compact Lie group acts on the manifold in a Hamiltonian fashion it is possible to define a moment map which takes values in the dual of the Lie algebra of the acting group; in particular when the group is a torus  $T^m$ and the action is effective the moment map takes values in $\R^m$. These manifolds are called  symplectic toric manifolds.
We consider $4$-plectic manifolds, i.e. $4m$-dimensional real manifolds endowed with a non degenerate closed $4$-form, acted on by the group $\Sp(1)^m$ in a {\em generalized Hamiltonian} fashion  so that it is possible to define
 a {\em tri-moment map} which takes values in $(\Lambda^3{\mathfrak{sp}(1)^*}^m)\cong \R^{m}$. 
Inspired by the symplectic setting we give the following
\begin{definition}  Let $M$ be a connected, compact $4m$-dimensional  $4$-plectic manifold on which  $\Sp(1)^m$ acts  effectively in a generalized Hamiltonian fashion with discrete principal isotropy. Then $M$ is called a {\em quaternionic toric manifold.}
 \end{definition}
The celebrated  Atiyah's convexity Theorem, \cite{At}, establishes the convexity of the image of the moment map for symplectic toric manifolds.\\
In some cases we are able to prove that the image of the tri-moment map is a convex polytope. More generally,  
when  a $4$-plectic manifold $(M, \psi)$,  acted on in a generalized Hamiltonian fashion by $Sp(1)^m$ with tri-moment map $\s,$ is equipped with 
a  {\em partitioned strongly non degenerate}  form $\psi$,
  we prove  that $\s(M)$  is contained in the convex envelope of a finite set of points, see Theorem \ref{convexity}.\\
In the other direction, in the symplectic setting the Theorem of Delzant proves that there is a one-to-one correspondence between symplectic toric manifolds and a special class of polytopes, the Delzant polytopes; in particular  in its well known paper \cite{Del} the author provides  a procedure to recover the  symplectic manifold starting from a Delzant polytope.
In the quaternionic setting the idea of defining toric manifolds starting from polytopes can already  be found in  \cite{S}  where the author begins the  study of  a new class of topological spaces analogous to real and complex toric varieties, but with the underlying structure provided by the skew field of quaternions. 
Starting from a $m$-dimensional convex polytope $P$  and a {\em characteristic function}  he defines a {\em quaternionic toric variety} to be a certain topological quotient of $P\times (S^3)^m$, where $S^3$ is the unit sphere of the quaternionic space $\H$. The author emphasizes that these are not algebraic varieties, and moreover he observes that  the notion of  ``quaternionic variety" is unclear, because quaternions are non-commutative and general polynomials are not well behaved. The author studies the topology  and the homology Betti numbers of the resulting objects. \\ In the present paper we introduce a procedure that, starting from a $m$-dimensional  Delzant polytope with appropriate hypotheses, leads to obtain a compact manifold acted on  effectively and with trivial principal isotropy by  $Sp(1)^m$. 
The advantage of our construction 
is that it  suggests a way, in the spirit of the symplectic cut \cite{Ler}, to equip the resulting manifold with a non-degenerate $4$-form. \\
Indeed we define the $4$-plectic cut as follows. Let $(M,\psi)$ be a $4$-plectic manifold equipped with a generalized Hamiltonian $Sp(1)^m$-action. Consider the restricted $Sp(1)$-action, and let $h:M\rightarrow \R$ be the corresponding  tri-moment  map. Let $\varepsilon$ be a regular value of $h$. 
For simplicity we assume that the $Sp(1)$-action on $h^{-1}(\varepsilon)$ is free. We denote by $M_{h>\varepsilon}, M_{h\geq \varepsilon}$ the pre-images of $(\varepsilon,\infty)$ and $ [\varepsilon,\infty)$ under $h : M\rightarrow \R$, and denote by $\overline{M_{h\geq\varepsilon}}$ {\em the $4$-plectic cut }, i.e. the manifold which is obtained by collapsing the boundary ${h^{-1}}(\varepsilon)$  of $M_{h\geq \varepsilon}$ along the orbits of the $Sp(1)$-action. \\
Let $\psi_0$ be the standard $4$-plectic form on $\H$. With the above notations we prove
\begin{theorem} \label{TEO7.2}  Let $(M,\psi)$ be a $4$-plectic manifold. Whenever the induced form $\psi\oplus \psi_0$ on $M\times \H$ is {\em horizontal} along ${(h-\frac{1}{4}|q|^4)}^{-1}(\varepsilon),$ there is a natural $4$-plectic structure $\Psi_{\varepsilon}$ on $\overline{M_{h\geq\varepsilon}}$ such
that the restriction of $\Psi_{\varepsilon}$ to $M_{h>\varepsilon} \subseteq \overline{M_{h\geq \varepsilon}}$ equals $\psi$ .
\end{theorem}
\noindent As an application we
find a correspondence between a special class of Delzant polytopes and some quaternionic toric manifolds. In these cases we are also able to show that the involved  manifolds admit an action of $(\H^*)^m$ with an open dense orbit, in analogy with what happens in the complex setting.
%
We observe that all these examples are {\em quaternionic regular} manifolds in the sense of \cite{GGS}.

The paper is organized as follows. In the second section we give the basic definitions and notions of $4$-plectic manifolds and generalized Hamiltonian actions.
In Section 3  we present a sub-convexity result, Theorem \ref{convexity}, and in Section 4 we describe the above mentioned procedure,  providing necessity and sufficiency conditions under which it can be applied. 
Section 5 is devoted to study the $4$-plectic reduction and the $4$-plectic cut. In the last section we collect the obtained examples, we give the explicit form of the tri-moment map and consequently we obtain a convexity theorem for this class of examples. We finally make some remarks on the $\H^*$ action and on the manifold $G_2/SO(4)$ which deserves further investigation.


\section{The $4$-plectic viewpoint}
\noindent We begin  this section by introducing  a possible counterpart of symplectic forms and  structures on $4m$-dimensional real manifolds.
\begin{definition} Let $M$ be  real manifold of dimension $4m.$ A  $4$-form $\psi$ on $M$ is said  to be {\em $4$-plectic}  if 
\begin{enumerate}
\item $\psi$ is closed, i.e. $d\psi=0$;
\item  $\psi$ is {\em non-degenerate},  i.e. the map $v\mapsto \iota_v\psi$ that contracts $\psi$ along a tangent vector field $v$ has trivial kernel.
\end{enumerate}
A $4$-plectic form  defines a {\em $4$-plectic structure} on $M$, and $M$  equipped with such a form is  called a {\em $4$-plectic manifold}.
\end{definition} 
A natural class of examples of $4$-plectic manifolds is given by $4m$-dimensional symplectic manifolds. Indeed
starting from a manifold with symplectic form $\omega$ we obtain a $4$-plectic manifold  by endowing it with the $4$-form $\omega\wedge \omega$. 
Quaternion K\"ahler manifolds equipped with the Kraines  form, \cite{K}, give a class of examples of non-symplectic $4$-plectic manifolds; a particularly large class of $4$-dimensional $4$-plectic manifolds is given by the Kulkarni four-folds, \cite{Ku}.\\ An interesting basic example that we will use in the sequel is the quaternionic space $\H^m,$ naturally identified with $\mathbb{R}^{4m},$ endowed with  the  $4$-plectic form $\psi_0$ defined by $$\psi_0=\sum_{i=1}^m dx_{4i-3}\wedge dx_{4i-2}\wedge dx_{4i-1}\wedge dx_{4i}$$ where $x_1,\ldots,x_{4m}$  are the coordinates  on $\R^{4m}$. Note that  $(\mathbb{H}^m,\psi_0)$ is not symplectic for $m>1$; indeed the form $\psi_0$ cannot be obtained as the square of a symplectic form. \\
The notion of Hamiltonian action in the symplectic setting is very useful and powerful. Indeed whenever a Lie group $G$ acts on a symplectic manifold in a Hamiltonian fashion it is possible to define a map $\mu:M\rightarrow \mathfrak{g}^*$, commonly known as moment map, which encodes many geometric information on the manifold and on the action. Whenever the action of the group $G=Sp(1)^k$ on  a $4$-plectic  $4m$-dimensional manifold $M$ is {\em generalized Hamiltonian}, our aim is  to define an analog  of the moment map also in the $4$-plectic setting, following the path indicated by Foth in \cite{F}. \\ 
Whenever a Lie group $G$ acts on a manifold $M$, it is possible to define a canonical map $\mathfrak{g}\rightarrow \Gamma({M, T M})$ which sends the vector $X\in \mathfrak{g}$ to the {\em fundamental vector field} $\widehat X$ in $M$, such that  at a point $p\in M,$ $$\widehat X_p=\frac{d}{dt}_{|_{t=0}}\exp tX\cdot p$$  Now, if $M$ is equipped with a $4$-plectic form $\psi, $  there is also a natural map $\mathfrak{g}\rightarrow A^3(M)$ which sends the generic vector $X\in \mathfrak{g}$ to the contraction of $\psi$ along $\widehat X$, i.e. to a $3$-form on M. Given  a tangent vector field  $Y\in \Gamma(M,TM)$, if the $3$-form $\iota_Y \psi$ is closed we say that $Y$ is a {\em locally Hamiltonian vector field}; if moreover $\iota_Y \psi$ is exact we say that $Y$ is a {\em Hamiltonian vector field}. \\
From now on we assume $M$ is  a $4m$-dimensional real manifold.

\begin{definition}  Let $(M,\psi)$ be a $4$-plectic manifold  on  which the group $Sp(1)^k$  acts preserving $\psi$. We say that  this action is {\em generalized Hamiltonian} if  for any $X\in\mathfrak{sp}(1)^k$  the fundamental vector field $\widehat{X}$ is Hamiltonian.\end{definition}
\noindent The standard basis of $\mathfrak{sp}(1)\cong\mathfrak{su}(2)$ is given by $$H=\left(
\begin{array}{cc}
i&0  \\
0&-i\\
\end{array}
\right), \;\;\;X=\left(
\begin{array}{cc}
0&1  \\
-1&0\\
\end{array}
\right),\;\;\;Y=\left(
\begin{array}{cc}
0&i  \\
i&0\\
\end{array}
\right),$$
which represent respectively the quaternion imaginary units $i,j,k$.\\
The space of $3$-vectors $\Lambda^3(\mathfrak{sp}(1))$ can be  identified with $\R$ by the isomorphism that sends $ X\wedge Y\wedge H\mapsto 1$.\\
Any $$\delta=(\delta_1,\delta_2,\ldots,\delta_k)=(U_1\wedge V_1\wedge W_1,\ldots,U_k\wedge V_k\wedge W_k)\in (\Lambda^3\mathfrak{sp}(1))^k$$ induces a $k$-tuple of $3$-vector fields  on $M$ $$\widetilde{\delta}=(\widetilde\delta_1,\widetilde\delta_2,\ldots,\widetilde\delta_k)=(\widehat U_1\wedge \widehat V_1\wedge \widehat W_1,\ldots, \widehat U_k\wedge \widehat V_k\wedge \widehat W_k)\in (\Lambda^3(TM))^k.$$
\begin{definition} \label{trimoment}Let  $Sp(1)^k$ act on a $4$-plectic manifold  $(M,\psi)$ in a generalized Hamiltonian fashion.  A tri-moment map $\s$ is a map
$$\s:M\rightarrow {((\Lambda^3{\mathfrak{sp}(1)})^k)}^*\cong{((\Lambda^3{\mathfrak{sp}(1)})^*)}^k\cong \R^{k}$$ satisfying the following conditions:
\begin{enumerate}
\item $\s$ is $Sp(1)^k$-invariant, i.e. $\s(g\cdot p)=\s(p)$;
\item for any $\delta=(\delta_1,\delta_2,\ldots,\delta_k)$ in $(\Lambda^3{\mathfrak{sp}(1)})^k$, $p\in M$ and $v\in T_pM$ we have
$$d \s_p(v) (\delta)=\sum_{i=1}^k  \iota_{\widetilde{\delta_i}_{p}} \psi(v)=:\iota_{\widetilde \delta_p}\psi (v)\;$$
where $\widetilde{\delta}_i$ is the tri-vector field induced by $\delta_i$.
\end{enumerate}
\end{definition}
\noindent  Since the coadjoint action of $Sp(1)$ on $\mathfrak{sp}(1)^*$ induces the trivial action on $(\Lambda^3\mathfrak{sp}(1))^*$, the action is indeed equivariant.\\
A further property of the tri-moment map is that  for any $p\in M$ such that $\s( p)$ is regular, and $V=(\Lambda^3(T_p(Sp(1)^k\cdot p)))^k$   $$\ker \ d\s_p={V}^{\perp_{\psi_p}};$$
In fact $${\ker} \ d\s_p=\{v\in T_p M\ | 0=d\s_p(v)(\delta)=\iota_{\widetilde \delta_p}\psi (v)\ \text{for any} \ \delta \in   (\Lambda^3\mathfrak{sp}(1))^k\}$$ equals 
$${V}^{\perp_{\psi_p}}=\{v\in T_p M \ | \ \iota_{\widetilde \delta_p}\psi (v)=0 \ \text{for any} \ \delta \in   (\Lambda^3 (T_p(Sp(1)^k\cdot p))^k\},$$
since  $(\Lambda^3\mathfrak{sp}(1))^k$ and $(\Lambda^3 (T_p(Sp(1)^k\cdot p))^k$ are isomorphic.

\noindent Notice that the tri-moment map $\s$ is defined up to a constant $C\in \R^k$; the $Sp(1)^k$-invariance of $\s$ implies that there is no further hypothesis on $C$. \\
We  are now ready to prove
\begin{proposition}  Let $(M,\psi)$ be  a $4$-plectic manifold acted on by $Sp(1)^k$ in a generalized Hamiltonian fashion. The tri-moment map exists 
 when $b_1(M)=0$.
 \end{proposition}  
 \begin{proof} It is enough to prove the statement for  a component of the tri-moment  map. Consider the standard basis $X,Y,H$ of the Lie algebra $\mathfrak{sp}(1)$ defined above.   Recall that  $$[X,Y]=2H,\ \; [Y,H]=2X \;\mbox{and}\ \;[H,X]=2Y.$$  The element  $$\delta_i=(0,0,\ldots,0,\overbrace{X\wedge Y\wedge H}^{i-\textrm{th\; component}},0,\ldots,0)\in ({\Lambda^3(\mathfrak{sp}(1))})^k$$ is identified with the usual canonical basis vector $e_i\in \R^k$.  For each $i$-th component of the tri-moment map, and any tangent vector $Z$   we have
 $$<d\s(Z),\delta_i>=d\s_i(Z)=\psi(\widehat{X}, \widehat{Y},\widehat{H},Z)=\iota_{\widehat X}\iota_{\widehat Y}\iota_{\widehat H}\psi(Z)$$
 So a  necessary and sufficient condition for the existence of  a map  satisfying condition (2)  of Definition \ref{trimoment} is that the $1$-form $\iota_{\widehat X}\iota_{\widehat Y}\iota_{\widehat H}\psi$ is closed.
 We prove the closedness of this form by applying  the well known relations  involving the Lie derivative $\mathcal{L}$  and  valid for any  tangent vector fields $U,W$ 
 \begin{equation}\label{1}\mathcal{L}_U =d\iota_U+\iota_Ud
 \end{equation} and 
 \begin{equation}\label{2} [\mathcal{L}_{{U}},\iota_{{W}}]=\iota_{[{U},{W}]}.\end{equation}
 We compute
$$d(\iota_{\widehat X}\iota_{\widehat Y}\iota_{\widehat H}\psi)=-\iota_{\widehat X}d\iota_{\widehat Y}\iota_{\widehat H}\psi+\mathcal{L}_{\widehat X}\iota_{\widehat Y}\iota_{\widehat H}\psi=-\iota_{\widehat X}d\iota_{\widehat Y}\iota_{\widehat H}\psi+\iota_{\widehat Y}\mathcal{L}_{\widehat X}\iota_{\widehat H}\psi+\iota_{[\widehat{X},\widehat{Y}]}\iota_{\widehat{H}}\psi$$
Now the last term of the equality is zero since $[X,Y]=2H$. Applying again  equation (\ref{2}) the equality above becomes
$$
-\iota_{\widehat X}d\iota_{\widehat Y}\iota_{\widehat H}\psi+\iota_{\widehat Y}\mathcal{L}_{\widehat X}\iota_{\widehat H}\psi=-\iota_{\widehat X}d\iota_{\widehat Y}\iota_{\widehat H}\psi+\iota_{\widehat Y}\iota_{\widehat H}\mathcal{L}_{\widehat X}\psi+\iota_{\widehat{Y}}\iota_{[\widehat{X},{\widehat{H}}]}\psi$$
Both the last two terms are zero: $\mathcal{L}_{\widehat X}\psi=0$ since the action is via $4$-plectomorphisms, and   $[H,X]=2Y$.
Applying twice equation (\ref{1}) and once equation (\ref{2})  and using the closeness of $\psi$ we get that 
$-\iota_{\widehat X}d\iota_{\widehat Y}\iota_{\widehat H}\psi$ equals
$$\iota_{\widehat X}\iota_{\widehat Y}d\iota_{\widehat H}\psi-\iota_{\widehat X}\mathcal{L}_{\widehat Y}\iota_{\widehat H}\psi=\iota_{\widehat X}\iota_{\widehat Y}\mathcal{L}_{\widehat H}\psi-\iota_{\widehat X}\iota_{\widehat Y}\iota_{\widehat H}d\psi-\iota_{\widehat X}\iota_{\widehat H}\mathcal{L}_{\widehat Y}\psi-\iota_{\widehat{X}}\iota_{[\widehat{Y},{\widehat{H}}]}\psi=0
$$
Now by averaging over the group, which is compact, we get an invariant tri-moment map.
\end{proof}
 \begin{proposition} \label{moment} The  multiplicative action of $Sp(1)^m$ on $(\H^m,\psi_0)$ given by $$(\lambda_1,\lambda_2,\ldots,\lambda_m)(q_1,q_2,\ldots,q_m):=(\lambda_1\cdot q_1,\lambda_2\cdot q_2,\ldots,\lambda_m\cdot q_m)$$ is generalized Hamiltonian, and the tri-moment map is given by $$\s(q_1,q_2,\ldots,q_m)=-\frac{1}{4}(|q_1|^4,|q_2|^4,\ldots,|q_m|^4)+C$$
\end{proposition} 
\begin{proof} It is sufficient to compute the  tri-moment map for $m=1$ since each factor of $Sp(1)^m$ acts on each $\H$ separately. We firstly determine the fundamental vector fields,  starting from $H,X,Y$ previously defined, at $q_1=x_1+ix_2+jx_3+kx_4\in \H$:
$$\widehat H_{q_1}=-x_2\frac{\partial}{\partial x_1}+x_1\frac{\partial}{\partial x_2}-x_4\frac{\partial}{\partial x_3}+x_3\frac{\partial}{\partial x_4};$$
$$\widehat X_{q_1}=-x_3\frac{\partial}{\partial x_1}+x_4\frac{\partial}{\partial x_2}+x_1\frac{\partial}{\partial x_3}-x_2\frac{\partial}{\partial x_4};$$
$$\widehat Y_{q_1}=-x_4\frac{\partial}{\partial x_1}-x_3\frac{\partial}{\partial x_2}+x_2\frac{\partial}{\partial x_3}+x_1\frac{\partial}{\partial x_4}.$$
Denoting by $e_{ijk}=\frac{\partial}{\partial x_i}\wedge\frac{\partial}{\partial x_j}\wedge\frac{\partial}{\partial x_k}$ the $3$-vector $\widehat{H}\wedge\widehat X\wedge\widehat Y$ at $q_1$ is therefore given by
$$ (-x_4 e_{123}+x_3 e_{124}+x_1 e_{234}-x_2 e_{134})\cdot \sum_{i=1}^4 (x_i)^2$$
Thus, the contraction of $\psi_0$ along $\widehat{H}_{q_1}\wedge\widehat X_{q_1}\wedge\widehat Y_{q_1}$, is
$$\iota_{\widehat{H}_{q_1}\wedge\widehat X_{q_1}\wedge\widehat Y_{q_1}}\psi_0=-|q_1|^2(\sum_{i=1}^4 x_i dx_i).$$
By  the definition of the tri-moment map  we get that 
$$\iota_{\widehat{H}_{q_1}\wedge\widehat X_{q_1}\wedge\widehat Y_{q_1}}\psi_0(\cdot)=d\s_{q_1}(\cdot)(H\wedge X\wedge Y)=\sum_{i=1}^4 \frac{\partial \s}{\partial x_i} dx_i$$ thus we conclude that  the first component of the moment map is $$\s(q_1)=-\frac{1}{4}|q_1|^4+C_1,$$and so we get the claim.
\end{proof}
\noindent Note that also  $(\lambda,q)\mapsto q\cdot \lambda^{-1}$ defines a multiplicative action of $Sp(1)$ on $\mathbb{H}$ which is still generalized Hamiltonian with the same tri-moment map $\sigma(q)=-|q|^4/4+C$.
Hence all the actions of the type
$$(\lambda_1,\lambda_2,\ldots,\lambda_m)(q_1,q_2,\ldots,q_m):=(\lambda_1\cdot q_1,\ldots,\lambda_k\cdot q_k,q_{k+1}\cdot \lambda_{k+1}^{-1},\ldots,q_m\cdot\lambda_m^{-1} )$$ are generalized Hamiltonian as well.\\
\noindent Inspired by the definition of toric symplectic manifolds we give the following
\begin{definition}\label{toric}  Let $M$ be a connected, compact, $4m$-dimensional, $4$-plectic manifold on which  $\Sp(1)^m$ acts effectively in a generalized Hamiltonian fashion with discrete principal isotropy. Then $M$ is called a {\em quaternionic toric manifold.}
 \end{definition}

\section{Towards a Convexity theorem}\label{convex} \noindent In this section  we prove a theorem  on the {\em sub-convexity} of the image of the tri-moment map.  This can be done under some additional hypotheses on the $4$-plectic form. The Darboux Theorem, that allows a local canonical expression for symplectic forms, does not  hold in general for $4$-plectic forms. Our assumptions in Theorem \ref{convexity} compensate for this lack.
A general convexity result can be proven for quaternionic Flag manifolds, \cite{F},   and in particular for quaternionic projective spaces, quaternionic Grassmannians and moreover for the Blow-ups of $\mathbb{HP}^m$ (see Remark \ref{conv}). We do not know if a convexity result holds for every $4$-plectic manifold acted on by $Sp(1)^k$.
Recall that in Atiyah's proof of  the convexity of the image of the moment map for a toric manifold, \cite{At}, a key ingredient is the fact that the moment map is a Morse-Bott function, i.e. a function with non-degenerate Hessian at the critical points. This does not hold in general for the tri-moment map; indeed for example the tri-moment map  $q\mapsto {-|q|^4}/{4}$ for the action of $Sp(1)$ on $(\H,\psi_0)$ has degenerate Hessian at the critical point.
 Moreover observe that even if the components of the tri-moment map are {\em minimally degenerate} in the sense of Kirwan \cite{Kir}, we cannot conclude that the image is convex. To see this consider the usual moment map of the complex projective space $\C\P^2$,  whose image is the standard simplex $\Delta_2\subseteq \R^2$; taking the square of the components of this map (which are Morse-Bott)  we get minimally degenerate functions (see p.290  in \cite{Kir}), whose image is not convex, but still contained in the convex envelope of the three vertices of the simplex $\Delta_2$.\\

 \noindent Take $(M,\psi)$ a $4$-plectic $4m$-dimensional manifold acted on in a generalized Hamiltonian fashion by $Sp(1)^m$ with tri-moment map $\s.$  Denote by $\s_i$  for $i=1,\ldots,m $ its components.  The critical set of each $\s_i$ will be denoted by $C_i={\rm{Crit}}\ \s_i$. The function $\s_i$ is constant on each connected component of 
${\rm{Crit}}\ \s_i.$\\
Consider, for each point $p\in M$, the set of linearly independent tangent vectors $\{\widehat{H}_p^1,\widehat{X}_p^1,\widehat{Y}_p^1,\ldots,\widehat{H}_p^m,\widehat{X}_p^m,\widehat{Y}_p^m\}$ where for each $H^i, X^i, Y^i$ in  the $i$-th term of the Lie algebra $\oplus_{i=1}^m \mathfrak{sp}(1)$ we have defined the fundamental vector fields  $\widehat{H}^i,\widehat X^i,\widehat Y^i$.
 We assume that it is possible to decompose the tangent space at each point $p\in M$ as the direct sum of  $m$ $4$-dimensional subspaces  $\{V_i\}_{i=1}^m$  with $\widehat{H}_p^i,\widehat{X}_p^i,\widehat{Y}_p^i\subseteq V_i$ in such a way that the restriction of the form $\psi$ to $V_i$ is non degenerate. In this case we call $\psi$ {\em strongly non degenerate}. Under this assumption it is possible to  define an isomorphism $L_{\psi}:\Lambda^3(V_i)\to V_i$ that allows  to construct a basis of $T_p M$ given by 
 $\{\widehat{H}_p^i,\widehat{X}_p^i,\widehat{Y}_p^i,\widetilde{\delta_p^i}\}_{i=1}^m=\{v_1^i,v_2^i,v_3^i,v_4^i\}_{i=1}^m$   where $\widetilde{\delta_p^i}=L_{\psi}(\widehat{H}_p^i\wedge \widehat{X}_p^i,\wedge \widehat{Y}_p^i)$ for all $i$.\\\\The critical points of each component of the tri-moment map $\s_i$  can be easily found using the defining properties of $\s$. Indeed a critical point $q$ of $\s_i$ is such that
 $$0=d\s_{i}(q)(v)=\iota_{\widetilde{\delta}^i}\psi_q(v)=\psi_q(v,\widehat{H}^i_q,\widehat {X}^i_q,\widehat {Y}^i_q)$$ for all $v\in T_q M$. Therefore  $q$ is a fixed point of at least one of  the $1$-parameter subgroups generated by $H^i,X^i$ or $Y^i$.
Hence a point $q$ in $M$ is critical for each component $\sigma_i$ if and only  if the isotropy of $q$ has maximal rank (i.e. contains a maximal torus $T\subseteq \Sp(1)^m$). The previous fact implies that $\cap_{i=1}^m C_i$ is a closed submanifold of the compact manifold $M$, and therefore it has a finite number of connected components.
 Since $\sigma$ is constant on the connected components of $\cap_{i=1}^{m} C_i$ , then $A=\sigma(\cap_{i=1}^{m} C_i)$  is a finite set.
Our aim is, now, to show that the image of $M$ via $\s$ is contained in the convex envelope of $A$, ${\rm Conv}(A)$ . The key ingredient, inspired by \cite{Kir}, is the following 
 \begin{lemma} \label{lemmaA} Let $F=(f_1,\ldots,f_m)$  with $f_i$ real-valued differentiable functions on a compact manifold $M$ such that their gradient vector fields  commute.  Assume that ${\rm{Crit}}\ f_i$ is a  submanifold for $i=1,\ldots,m$. Setting
 $B=F\left(\cap_{i=1}^{m} {\rm{Crit} }\ f_i\right)$, then we have $F(M)\subseteq \rm{Conv} (B).$
 \end{lemma} 
 \begin{proof} In view of the Hyperplane Separation Theorem, it is sufficient to prove that, for a generic $(\lambda_1,\lambda_2,\ldots,\lambda_m)\in \mathbb{R}^m,$ the restriction  of the  linear functional $y \mapsto  \sum_{i=1}^{m}\lambda_i y_i,$ to $F(M)$ takes its maximum value at a point of $B$. Equivalently, it is enough to prove that, for a generic  $(\lambda_1,\lambda_2,\ldots,\lambda_m)$, the maximum of the differentiable function $\varphi=\sum_{i=1}^{m}\lambda_i f_i$ defined on $M$   is attained in $Z=\cap_{i=1}^{m} {\rm{Crit} }\ f_i$. 
 The set $Z$ of common critical points of $f_1,\ldots,f_m$ is also the fixed-point set of the torus $T$ generated by the fields ${\rm{grad}} f_1,\ldots, {\rm{grad}} f_m$. 
 Moreover if $\sum_{i=1}^m\lambda_i f_i$ is a generic linear combination, so that the corresponding gradient vector field generates T, then its critical set  is precisely $Z$, and in particular $\varphi$ takes its maximum on $Z$.
\end{proof}
 Let us  assume, as a further hypothesis, that the form $\psi$ is zero whenever computed on at least two vectors $v^i\in V_i,$ and $v^j\in V_j$ belonging to two different $4$-dimensional subspaces of the tangent space $T_pM$. Then we can define a diagonal metric $g$ in terms of $\psi$ in the following way:
 $$
 g_p(v_j^i,L_{\psi}(v_t^k\wedge v_s^k \wedge v_r^k))=\psi_p(v_j^i,v_t^k,v_s^k,v_r^k)$$
 for any $i,k=1,\ldots,m$ and $j,t,s,r=1,\ldots,4.$
Hence
$$g_p(v_j^i,L_{\psi}(v_t^k\wedge v_s^k \wedge v_r^k))=0\;\;{\rm{for}}\;i\neq k\;\;{\rm{or}}\; \;j=t,s,r.$$
Observe that
the metric $g$ is non degenerate since $\psi$ is strongly non degenerate. 
Whenever the form $\psi$ satisfies all the previous assumptions we call it {\em partitioned strongly non degenerate}.
 
With the previous notations and with an additional algebraic assumption, we can prove
 \begin{theorem} \label{convexity} Let $(M, \psi)$ be a compact, connected, $4m$-dimensional  manifold equipped with a partitioned strongly non degenerate $4$-plectic form,  acted on in generalized Hamiltonian fashion by $Sp(1)^m$ with tri-moment map $\s=(\s_1,\ldots,\s_m)$. Suppose that  $v_4^j(g_p(v_4^k,v_4^k))=0$ for all $j,k=1,\ldots,m$. Then $\s(M)$  is contained in the convex envelope of the points of $A=\s\left(\cap_{i=1}^{m} {\rm{Crit} }\ \s_i\right)$.\end{theorem}
 \begin{proof} In order to prove the theorem we show that 
 $\left[{\rm{grad} }\s_i,{\rm{grad}}\s_j\right]=0$ for all $i,j$. 
We fix  a basis in $\R^m\cong (\Lambda^3\mathfrak{sp}(1))^m$ given by $\delta^1,\delta^2,\ldots,\delta^m$ where $\delta^i=H^i\wedge X^i\wedge Y^i$ for all $i$.  We have that
the differential of the  $i$-th component of the tri-moment map is such that
$$(d\s_{i})_p(v)=d\s_p(v)({\delta^i})=\iota_{v_1^i}\iota_{v_2^i}\iota_{v_3^i}\psi_p(v)=\psi_p(v,v_1^i,v_2^i,v_3^i)=g_p(v,v_4^i)=g_p(v,\widetilde{\delta^i}_p)$$
so that ${{(\rm{grad }}\s_i)}_p=\widetilde{\delta^i}_p$ for all $i=1,\ldots,m$, and 
$$[{\rm{grad}} \s_i,{\rm{grad} }\s_j]=[\widetilde{\delta^i},\widetilde{\delta^j}].$$
Let us now prove that $[\widetilde{\delta^i},\widetilde{\delta^j}]=0,$ i.e. that $$\iota_{[\widetilde{\delta^i},\widetilde{\delta^j}]}\psi_p(a,b,c)=0$$ for any three vectors $a,b,c\in T_p M$. 
First observe that if $a,b,c$  do not belong to the same $V_k$ in the decomposition of the tangent space or if $k\neq i,j$   then the equality is trivially true. Then, assuming that $k=i$ there are two possibilities:
without loss of generality, either  $a=\widetilde{\delta^i}$ or $(a,b,c)=(v_1^i,v_2^i,v_3^i)$.
The fact that $g$  is diagonal combined  with compatibility of the metric with  the Levi Civita connection, gives in the first case $\iota_{[\widetilde{\delta^i},\widetilde{\delta^j}]}\psi_p(a,b,c)=0$,
while in the second $\iota_{[\widetilde{\delta^i},\widetilde{\delta^j}]}\psi_p(a,b,c)=-\widetilde{\delta}^j g_p(\widetilde{\delta}^i, \widetilde{\delta}^i )=v_4^j(g_p(v_4^i,v_4^i))$ which vanishes by the hypothesis.
We can now apply Lemma \ref{lemmaA} and get the claim.
\end{proof}

\noindent In the particular case of a toric quaternionic manifold, we can find a bound on the cardinality of the set $A$. To prove this we give the following definition
\begin{definition} Let $G$ be a compact Lie group acting on a manifold $M$ with principal isotropy $G_{princ}.$ We call the {\em homogeneity rank}  the numerical invariant of the action
$$hrk(M,G):=\rank(G)-\rank(G_{princ})-\dim M+\dim G-\dim (G_{princ})$$ where $\rank(G)$ is the dimension of a maximal torus $T$ in $G$.
\end{definition}
\begin{theorem}\cite{Br} If $M^T$ is not empty and $hrk(M,G)\leq 0$ then $M^T$ is finite, and moreover its cardinality is equal to the Euler Characteristic of $M$.
\end{theorem}
\noindent Note that for quaternionic toric manifolds the homogeneity rank is $0$, thus the fixed point set $M^T$ is finite and $\#M^T=\chi(M)$.
\begin{remark} The maximal tori theorem states that, in a compact Lie group, all maximal tori are conjugate; so the cardinality of the set $M^T,$ when finite, does not depend
on the chosen torus $T$.
\end{remark}
\noindent In  particular for quaternionic  toric manifolds,  endowed with a strongly non degenerate $4$-plectic form, the cardinality of $A$ is bounded above an below respectively by the cardinality of the set $M^T$ of  points in $M$  fixed by a maximal torus $T\subset Sp(1)^m$, and the cardinality of the set $M^{\Sp(1)^m}$ of points fixed by  $\Sp(1)^m$.\\
 We have therefore proved the following
\begin{proposition} If $(M,\psi)$ is a quaternionic toric manifold, and $\psi$ is strongly non degenerate, the cardinality $\#A$ of $A$ satisfies the inequalities $$\#M^{\Sp(1)^m}\leq\#A\leq\#M^{T}.$$
\end{proposition}

\section{From polytopes to manifolds}
\label{poli}
The introduction of the  tri-moment map for the multiplicative action of $Sp(1)^n$ on $\H^n$ can be used  to construct, starting from appropriate  Delzant polytopes in ${\R^m}^*$,  real manifolds of dimension $4m$  acted on by  $Sp(1)^m$ with trivial principal isotropy. \\
We recall that a Delzant polytope has rationality, simplicity and regularity properties \cite{Del}. Our procedure is  inspired by the Delzant construction  that associates  a symplectic manifold with a polytope.
We recall the following 
\begin{definition} A Delzant polytope $P$ in ${\R^m}$ is a convex polytope such that
\begin{enumerate}
\item $P$ is {\em simple}, i.e., there are $n$ edges meeting at each vertex;
\item $P$ is {\em rational}, i.e., the edges meeting at a vertex $p$ are rational in the sense that each edge is of the form $p+tu_i,t\geq 0$ , where $u_i \in \Z^m$;
\item $P$ is {\em smooth}, i.e., for each vertex, the corresponding ${u_1, \ldots,u_m}$  can be chosen to be a $\Z$-basis of $\Z^m$.
\end{enumerate}
\end{definition}
\noindent In what follows, we will always consider polytopes of the dual space ${\R^m}^*$. Let $P\subset {\R^m}^*$ be a Delzant polytope with $d$ {\em facets}, i.e. $(m-1)$-dimensional faces. Let $v_i\in \Z^m$ with $i=1,\ldots ,d$ be the primitive outward-pointing normal vectors to the facets. For some $\lambda_i\in \R$ we can write
$$P=\{x\in {\R^m}^*|<x,v_i>\leq \lambda_i,\;i=1,\ldots,d\}.$$
Let $\{e_1,\ldots,e_d\}$ be the standard basis of $\R^d$. Consider the map $\pi:\R^d\rightarrow\R^m$ defined by $e_i\mapsto v_i$. By Lemma 2.5.1 in \cite{Sil} we know that the map $\pi$ is onto and maps $\Z^d$ onto $\Z^m$. Therefore $\pi$ induces  a surjective map, still called $\pi$, between tori \begin{equation}\label{tori}\pi:\T^d\rightarrow \T^m.\end{equation} The kernel  $N$ of $\pi$ is a $(d-m)$-dimensional Lie subgroup of $\T^d$, with inclusion map $i:N\rightarrow\T^d$.  The exact sequence of tori
$$1\rightarrow N\xrightarrow{i} \T^d\xrightarrow{\pi}\T^m\rightarrow 1$$ induces an exact sequence at the Lie algebra level
$$0\rightarrow \mathfrak{n}\xrightarrow{i} \R^d\xrightarrow{\pi}\R^m\rightarrow 0$$ with dual exact sequence
$$0\rightarrow {(\R^m)}^*\xrightarrow{\pi^*}{(\R^d)^*}\xrightarrow{i^*} \mathfrak{n}^*\rightarrow 0.$$\\
Now consider $\H^d$ with the $4$-plectic form $\psi_0$ and the standard generalized Hamiltonian action of $Sp(1)^d$. In Propsition \ref{moment} we have computed  the tri-moment map $\s$ for this action.
The subtorus $N={(S^1)}^{d-m}$ acts on $\mathbb{H}^d$. Our procedure works whenever the action of $N$ on $\mathbb{H}^d$ extends to $\widehat{N}=Sp(1)^{d-m}\cong(S^3)^{d-m}$. 
\subsection{Extendibility of the action of $N$} 
Due to the non-commutativity of  quaternions,  it is not always possible to extend the action of the subtorus $N$. In fact the only way to define a multiplicative action of $Sp(1)^{d-m}$ on $\mathbb{H}^d$ is the following
$$(h_1,\ldots,h_{d-m})(q_1,\ldots,q_d)=(h_{k_1}^{\alpha_1}q_1h_{j_1}^{\beta_1}, \ldots,h_{k_d}^{\alpha_d}q_dh_{j_d}^{\beta_d})$$
with $k_\ell,j_\ell \in \{1,\ldots, d-m \}$ and $\alpha_\ell\in \{0,1\} ,\ \beta_\ell \in \{0,-1\}$.

We  want to collect some necessary and sufficient conditions under which the action of $N$ can be extended to an action of $\widehat{N}$ on $\mathbb{H}^{d-m}$. We state these conditions in terms of the basis of the Lie algebra $\mathfrak{n}$ of $N$. 

\begin{prop}\label{Adm}
A necessary condition for the action of $N$ to be extendable to $\widehat N$ is that there exists a basis $\mathcal{B}=\{b_1,\ldots,b_{d-m}\}$ of $\mathfrak{n}\subset \R^d$ such that the $d\times (d-m)$ matrix $A_\mathcal{B}=(b_1|\ldots|b_{d-m}),$ whose columns are $b_1,\ldots,b_{d-m},$  has the following properties:

\begin{enumerate}
\item its only possible entries are $1,$$-1$ and $0$;
\item in each of its rows at most two entries are not zero.
\end{enumerate}
In this case we say that $\mathfrak{n}$ admits a {\em reduced} basis.
\end{prop}
\begin{proof}
Suppose that $d-m\ge 3$ and that  there  exists no basis of $\mathfrak{n}$ satisfying (2). Then for any $\mathcal{B}$ there exists at least one  row of $A_\mathcal{B}$ that has (at least) $3$ entries which are non-vanishing. 
 This means that  one of the defining equations of $\mathfrak{n}$  is of the form $x_i=\alpha x_{j}+\beta x_{k}+\gamma x_{l}$. In terms of $\widehat N$ this becomes a multiplicative relation involving $4$ unitary quaternions $h_i=h_{j}^{\alpha}h_{k}^{\beta}h_{l}^{\gamma}$, and  this equation does not allow to define a multiplicative action of $Sp(1)^{d-m}$ on $\HH^d$ since at least two (not commuting) factors have to be placed on the same side of $q_i$.\\
Suppose now that for any basis satisfying (2)  condition (1) is not fulfilled.   This means that  one of the defining equations of $\mathfrak{n}$  is of the form $x_i=\alpha x_{j}+\beta x_{k}$ with $\alpha$ or $\beta$ different from $0,1,-1$.
As before we have a relation between elements of $\widehat{N}$ of the form  $h_i=h_{j}^{\alpha}h_{k}^{\beta}$ which does not define a multiplicative action.

 \end{proof}
 \noindent A basis satisfying condition (1)  exists if and only if the starting polytope facets form angles   
 which are multiple of $\frac{\pi}{4}$, otherwise entries different from $1$ or $-1$ necessarily occur.  \\
The necessary condition  given in the previous proposition is not sufficient, indeed
\begin{example} Assume that $m=2$, $d=4$ and consider the polytope whose normals  are $\{(-1,0),(0,-1),(1,1,),(-1,1)\}$. The matrix $A_{\mathcal{B}}$ is given by  a $4\times 2$ matrix  whose  columns are $(1,1,1,0)$ and $(-1,1,0,1)$  and correspondingly $(h_1,h_2)(q_1,q_2,\cdot,\cdot)=(h_1q_1h_2,h_1q_2h_2^{-1},\cdot,\cdot)$ which does not define an action of $Sp(1)^2$ on $\mathbb{H}^4$. 
\end{example}
\noindent The previous fact is general, indeed by direct computation we can prove
\begin{theorem} \label{adm} Let $P$ be a Delzant polytope in ${\R^m}^*$ with $d$ facets, let  $\pi: T^d\rightarrow T^m$   be  defined as in (\ref{tori})  and  let $N=\ker \pi$.
The action of $N$ can be extended to an action of $\widehat{N}$ on $\mathbb{H}^{d}$ if  and only if one can find a reduced basis $\mathcal{B}$ of the Lie algebra $\mathfrak{n}$ of $N$ in such a way that the matrix $A_\mathcal{B}$ does not contain a sub matrix of the form 
$$
\mathcal{M}=\left(\begin{array}{cc}
1&-1\\
1&1\\
\end{array}
\right).
$$
\end{theorem}
\noindent Theorem \ref{adm} applies  in particular to polytopes obtained by cutting the standard  simplex by means of hyperplanes parallel to the facets of the simplex, but these do not cover all the polytopes for which the action of $N$ can be extended, indeed  
\begin{example} \label{Ex} Consider the polytope in  ${\R^3}^*$ whose primitive outward-pointing normals  to the facets are $$\{(0,-1,0),(0,0,-1),(0,1,0),(-1,0,0),(1,1,0),(1,1,1)\}.$$ In this case,   the vector  $(1,1,0)$  is not orthogonal to any facet of the simplex, but a reduced basis for $N$ given by $$\mathcal{B}=\{(1,0,1,0,0,0), (0,1,0,0,-1,1),(1,0,0,1,1,0)\}$$  is such that $A_{\mathcal{B}}$ does not contain the matrix $\mathcal{M}$. Hence, thanks to Theorem \ref{adm},  the action of $N$ can be extended. For any $n=(h_1,h_2,h_3)\in \widehat{N}$ and $q=(q_1,\ldots,q_6)$
\begin{equation}\label{esempio}n\cdot q:=(h_1q_1h_3^{-1},q_2h_2^{-1},h_1q_3,q_4h_3^{-1},h_2q_5h_3^{-1},q_6h_2^{-1}).
\end{equation}
\end{example}

\subsection{Towards the construction of the manifold} 
Let  $P$ be the Delzant polytope defined by$$P=\{x\in {\R^m}^*|<x,v_i>\leq \lambda_i,\;i=1,\ldots,d\}.$$
From now on we will assume that the action of $N$ can be extended to $\widehat {N}$. The tri-moment map for the action of $\widehat{N}\cong Sp(1)^{(d-m)}$ on $\H^d$ is given by  $i^*\circ\s$ where  $\s(q_1,q_2,\ldots,q_d)=-\frac{1}{4}(|q_1|^4,|q_2|^4,\ldots,|q_d|^4)+C,$
 for some $C=(C_1,\ldots,C_d),$ is the tri-moment map for the standard action of $\Sp(1)^d$ on $\HH^d$  and   $$i^*:{\R^{d}}^*\cong((\Lambda^3{\mathfrak{sp}(1))^d)}^*\rightarrow((\Lambda^3{\mathfrak{sp}(1))^{d-m}})^{*}\cong{\R^{(d-m)}}^* $$ is the dual of the inclusion  map. Now we choose the constant $C$ to be $(\lambda_1,\ldots,\lambda_d).$
\begin{lemma}  Let $Z= (i^*\circ\s)^{-1}(0)$. $Z$ is compact and, if we assume that the action of $N$  extends to an action of ${\widehat N}$, then the action of $\widehat N$ on $Z$ is free. 
\end{lemma}
\noindent The  proof of the first part of the statement is the same as in the symplectic case (see e.g. \cite{Sil}). We present  it here for the convenience of the reader.
\begin{proof} 
It is sufficient to show that $Z$ is bounded, since it is clearly closed.
 Let $P'=\pi^*(P)$. We want to show that $\s(Z)=P'$. Observe  that  if $y\in {\R^d}^*$ then $y\in P'$ if and only if $y$ is in the image of $Z$ via $\s$. Indeed we have that $y\in \s(Z)$ if and only if  
 \begin{itemize}
 \item $y$ is in the image of $\s$; 
 \item $i^*(y)=0$.
 \end{itemize}
 Now we use the expression of $\s$ and the fact that $\Im \:\pi^*={\ker}\; i^*$, so we have that these two conditions are equivalent to 
 \begin{enumerate}
 \item $<y,e_\ell>=y_\ell=-\frac{1}{4}|q_\ell|^4+\lambda_i\leq \lambda_\ell$ for all $\ell=1,\ldots, d$
 \item $y=\pi^*(x)$ for some $x\in {\R^m}^*$.
 \end{enumerate}
 From this we have, for all $\ell$
 $$<y,e_\ell>\leq \lambda_\ell\iff <\pi^*(x),e_\ell>\leq \lambda_\ell\iff$$$$<x,\pi(e_\ell)>\leq \lambda_\ell\iff <x,v_\ell>\leq \lambda_\ell\iff x\in P$$
thus $y\in \s(Z)\iff y\in \pi^*(P)$.  Thanks to the properness of $\s$ and the compactness of $P'$ we get that  $Z$ is bounded, and therefore compact.

A key ingredient to prove the freeness of the action of $\widehat{N}$ on $Z$ is the fact that $N$ acts freely on $Z$, as proven in Proposition 29.1 in \cite{Sil}. Consider a vertex  $p\in P$ and let $I=\{i_1,\ldots,i_m\}$ be the set of indices for the $m$ facets meeting at $p$.  Pick $q\in Z$ such that $\s(q)=\pi^*( p)$. Then $p$ is characterized by $m$ equations $<p,v_\ell> =\lambda_\ell$ with $\ell\in I$.
Now $$<p,v_\ell> =\lambda_\ell\iff<p,\pi(e_\ell)> =\lambda_\ell\iff<\pi^*( p),e_\ell> =\lambda_\ell$$
$$\iff<\s(q),e_\ell> =\lambda_\ell\iff-\frac{1}{4}|q_\ell|^4+\lambda_\ell=\lambda_\ell\iff q_\ell=0$$
 Hence those $q\in Z$ wich are sent to vertices  in $P$ are points whose coordinates in the set $I$ are zero. Without loss of generality, we can assume that $I=\{1,\ldots,m\}$ so that the stabilizer of $q$ is $(Sp(1)^d)_q=\{(h_1,\ldots,h_m,1,\ldots,1)\in Sp(1)^d\}\cong Sp(1)^m$. \\The group $\widehat{N}\cong Sp(1)^{d-m}$ acts on $Z$ and its action on  non-zero coordinates is of the form
 $h_{k_\ell}^{\alpha_\ell}q_\ell h_{j_\ell}^{\beta_\ell}$
 where $\alpha_\ell=0$ or $1$ and $\beta_\ell=0$ or $-1$.
 Each  $h_{k_\ell}, h_{j_\ell}$ for  $k_\ell,j_\ell\in\{1,\ldots,m\}$  acts on at least a  non-zero coordinate of $q$, since otherwise $N$ would not act freely on $Z$. \\
 If the action  on the $\ell$-th coordinate of $q$ is $q_\ell\mapsto h_{k_\ell}q_\ell  h_{j_\ell}^{-1}$ we say that  $({k_\ell},{j_\ell})$ form a  {\em couple}.
 We use all the couples to construct a graph $\Gamma$ with $m$ distinct  vertices labeled ${1,\ldots,m}$ in the standard way. We want to show that each connected component of $\Gamma$ contains a vertex $s$ such that $h_s$ acts as a simple multiplication on a non-zero coordinate of $q$.  On the contrary suppose that there is a connected component of $\Gamma$ whose vertices are $V_1,\ldots,V_t$ such that $h_{V_1},\ldots,h_{V_t}$  do not act as simple multiplication on any coordinate of $q$. This implies that setting $h_{V_1}=\ldots=h_{V_t}=-1$ identifies a non-trivial element in the stabilizer of $q$ in $N$, contradicting the  freeness of the action of $N$ on $Z$.
 This allows us to conclude that for each connected component of $\Gamma$ at least  a vertex $s$  is such that $h_s$ acts as a simple multiplication on a non-zero coordinate of $q$. Therefore
 $h_{s}=1$ and the same holds true for all $h_{V_i}$ in the same connected component.  
 Note that for all other $q'\in Z$, which are not sent to vertices,  the stabilizer is necessarly smaller and therefore trivial. 
\end{proof}
\noindent Thanks to the previous lemma we can consider the orbit space $$M_{P}=Z/\widehat{N}$$ which  is a compact manifold of (real) 
dimension $$\dim Z-3(d-m)=4d-(d-m)-3(d-m)=4m.$$

\noindent In the complex setting the group $N$ is normal in ${(S^1)}^{d}$ so that one can directly define the group $G$ that acts on $Z/N$ as the  quotient group  ${{(S^1)}^{d}}/{N}$. In the present setting we have to define appropriately the group $\widehat{G}$ as a $3m$-dimensional group, isomorphic to $Sp(1)^m$, acting on ${Z}/{\widehat{N}}$. And this is not always possible. 
In particular, this can be done when
 $\widehat{N}$ acts by simple multiplication (either on the right or on the left) on at least $m$ coordinates in $\HH^d$. 
We now state some conditions under which the action of $\widehat G$  is effective and has trivial isotropy.
\begin{proposition} \label{condsuffG} If a polytope in ${\R^m}^*$ is obtained
 by cutting the standard  simplex by means of hyperplanes parallel to the facets of the simplex, then the group $\widehat G$ can be defined in such a way that its action on ${Z}/{\widehat N}$ is effective and with trivial isotropy. 
\end{proposition}
\begin{proof} Since we start from the simplex in ${\R^m}^*$, the $m+1$ oriented vectors orthogonal to the facets  are given by $\{-e_1,-e_2,\ldots,-e_m,e_1+e_2+\cdots+e_m\}.$ If we cut our simplex with $m+1$ hyperplanes parallel to its facets, we can get a polytope with $d=2m+2$ facets. The matrix associated with $\pi$ is an $m\times (2m+2)$ matrix given by $(-Id_m,C_{m+1},-C_{m+1}, Id_m)$ where $C_{m+1}$ is the column vector $e_1+e_2+\cdots+e_m$ and $Id_m$ is the identity matrix of order $m$. The kernel $\mathfrak{n}\subseteq \R^d$ of $\pi$ is defined by the linear system of $m$ equations
\begin{equation}
\left\{\begin{array}{l}
x_1=x_{m+1}-x_{m+2}+x_{m+3}\\
x_2=x_{m+1}-x_{m+2}+x_{m+4}\\
\cdots\\
x_m=x_{m+1}-x_{m+2}+x_{2m+2}\\
\end{array}
\right.
\end{equation}
We now choose  the new parameters $x'_{m+1}=x_{m+2}-x_{m+1},x_{m+2},\ldots,x_{2m+2}$ instead of $x_{m+1},x_{m+2},\ldots,x_{2m+2}$  and we get
a  basis $\mathcal{B}$  for $\mathfrak{n},$ 
given by the $m+2$  vectors 
$$\{-f_1-f_2-\ldots -f_m+f_{m+1},f_{m+2},f_1+f_{m+3},f_2+f_{m+4},\ldots,f_m+f_{2m+2}\}$$ where $f_i$ are the elements of the standard basis in 
$\R^{2m+2}$. This basis is reduced and the corresponding $A_\mathcal{B}$ does not contain the sub matrix $\mathcal{M}$. Thanks to Proposition \ref{adm} we can extend the action of $N$ to $\widehat N$ on $\H^{2m+2}$.  
For a generic element $n=(h_{m+1},h_{m+2},\ldots,h_{2m+2})$ in $\widehat N$ and $q=(q_1,q_2,\ldots,q_{2m+2})$ in $\H^{2m+2}$ we have
$$n\cdot q:=(h_{m+3}q_1h_{m+1}^{-1},\ldots,h_{2m+2}q_mh_{m+1}^{-1},h_{m+2}q_{m+1}h_{m+1}^{-1},h_{m+2}q_{m+2},\ldots,h_{2m+2}q_{2m+2});$$
Then  we say that a generic element $g=(g_1,g_2,\ldots,g_m)$ of $\widehat G \subseteq Sp(1)^d$  acts on the orbit $(\widehat N\cdot q)$ as  
\begin{align*}
g\cdot (\widehat N\cdot q):= (h_{m+3}q_1h_{m+1}^{-1},&...,h_{2m+2}q_mh_{m+1}^{-1}, h_{m+2}q_{m+1}h_{m+1}^{-1},\\
& h_{m+2}q_{m+2}g_1^{-1},..,h_{2m+1}q_{2m+2}g_{m}^{-1},h_{2m+2}q_{2m+2}).
\end{align*}
It is clear that $\widehat G$ takes $\widehat N$-orbits to $\widehat N$-orbits. Moreover its action is effective and has trivial principal isotropy. Indeed
suppose the group $\widehat G$ fixes an orbit through a point $\widetilde q$, whose coordinates are all non-vanishing, thus  from the relation  $h_{2m+2}\widetilde {q}_{2m+2}=\widetilde {q}_{2m+2}$ we deduce that
 $h_{2m+2}=1$ and consequently that $h_{m+1}=h_{m+2}=\ldots=h_{2m+2}=1$.
To conclude note that the previous argument (slightly modified) can be applied also if  we cut by means of   less than $m+1$ hyperplanes.
\end{proof}

\noindent Note that  Example \ref{Ex} shows that the assumption in  Proposition \ref{condsuffG},  on the way the simplex is cut  to construct the
 polytope, is not necessary to define an action of $\widehat{G}$ on the quotient manifold. Indeed in this case the action of $\widehat{G}$ given by
 $$g\cdot \widehat{N}q:=(h_1q_1h_3^{-1},g_1q_2h_2^{-1},h_1q_3g_2^{-1},g_3q_4h_3^{-1},h_2q_5h_3^{-1},q_6h_2^{-1})
$$
has trivial principal isotropy (and hence is effective).

\subsection{Examples of the procedure}
 We now collect some useful examples of manifolds  constructed starting from  a given polytope.
 \begin{example}{\em  Starting from the standard simplex $\Delta_m$ in ${\R^m}^*$ we obtain the quaternionic projective space $\H \P^m$.
 Indeed, the kernel $\mathfrak{n}$ of $\pi:\R^{m+1}\to\R^m$ is defined by the linear system
 \begin{equation}
\left\{\begin{array}{l}
x_1=x_{m+1}\\
x_2=x_{m+1}\\
\cdots\\
x_m=x_{m+1}\\
\end{array}
\right.
\end{equation}
So the action of $\widehat{N}$ is given by
$$n\cdot q:=(hq_1,\ldots,h q_{m+1}),$$
and the corresponding  tri-moment map is 
$$i^*\circ \sigma(q):=-\frac{1}{4} \sum_{i=1}^{m+1}|q_i|^4+1.$$
Thus $Z=(i^*\circ \sigma)^{-1}(0)$ is diffeomorphic to the $(4m+3)$-dimensional sphere in $\R^{4m+4}$. Recalling the classical
 Hopf fibration $\widehat{N}\cong S^3\rightarrow S^{4m+3}\rightarrow \H \P^m$, we get that $M_{\Delta_m}=\H \P^m$.  The action of $\widehat{G}\cong Sp(1)^m$ can  be defined on $\widehat{N}$-orbits as
 $$g\cdot \widehat{N}q=(hq_1g_1^{-1},\ldots,hq_{m}g_m^{-1},h q_{m+1})$$
and it has clearly  trivial isotropy (and hence it is effective).}
 \end{example}
 \begin{example}{\em  If the polytope $P$ is a square $[0,1]\times[0,1]\in {\R^2}^*$
  the corresponding manifold $M_P$ is  $\H \P^1\times \H \P^1$. Indeed in this case $d=4$ 
  the kernel $\mathfrak{n}$ of $\pi:\R^{4}\to\R^2$ is defined by the linear system
 \begin{equation}
\left\{\begin{array}{l}
x_1=x_{3}\\
x_2=x_{4}\\
\end{array}
\right.
\end{equation}
So the action of $\widehat{N}$ is given by
$$n\cdot q:=(h_1q_1,h_2q_2, h_1q_3,h_2q_{4}),$$
and the corresponding  tri-moment map is 
$$i^*\circ \sigma(q):=\left(-\frac{1}{4} (|q_1|^4+|q_3|^4)+1,-\frac{1}{4} (|q_2|^4+|q_4|^4)+1\right).$$
Thus $Z=(i^*\circ \sigma)^{-1}(0,0)$ is diffeomorphic to the product of two $7$-dimensional spheres in $\R^{8}$
   on which $\widehat{N}\cong Sp(1)^2$ acts separately on each factor. Thus, using again the Hopf fibration $S^3\rightarrow S^7\rightarrow S^4$, we find that the quotient space is the above mentioned product.  The action of $\widehat{G}\cong Sp(1)^2$ which  can be defined on $\widehat{N}$-orbits as
 $$g\cdot \widehat{N}q=(h_1q_1g_1^{-1},h_2q_2g_2^{-1}, h_1q_3,h_2q_{4}),$$
has clearly  trivial isotropy (and hence  is effective).\\
This procedure can be naturally generalized to prove that, starting from the $m$-dimensional cube  $[0,1]\times\cdots\times[0,1]\in {\R^m}^*$, we find  the product of $m$-copies of $\H\P^1$ acted on by
$\widehat{G}\cong Sp(1)^m$ whose action  is defined on $\widehat{N}\cong{Sp(1)^m}$-orbits as
 $$g\cdot \widehat{N}q=(h_1q_1g_1^{-1},h_2q_2g_2^{-1},\ldots,h_mq_mg_m^{-1}, h_1q_{m+1},\ldots, h_mq_{2m}).$$
Note that the $m$-dimensional cube can be obtained by cutting the standard simplex  by means of $m$ hyperplanes parallel to the coordinate hyperplanes  of ${\R^m}^*$.}

 \end{example}
\begin{example}\label{trapezoid}  {\em We here start from the trapezoid $T$ in ${\R^2}^*$, defined by $$T=\{{x}\in {\R^2}^* | <{x},v_1>\leq 0, <{x},v_2>\leq 0, <{x},v_3>\leq 1, <{x},v_4>\leq 2\}$$ with $v_1=-e_2,v_2=-e_1,v_3=e_2$ and $v_4=e_1+e_2$. 
The kernel $\mathfrak{n}$ of $\pi:\R^{4}\to\R^2$ is defined by the linear system
 \begin{equation}
\left\{\begin{array}{l}
x_1=x_{3}+x_{4}\\
x_2=x_{4}\\
\end{array}
\right.
\end{equation}
so the action of $\widehat{N}\cong Sp(1)^2$ is given by
$$n\cdot q:=(h_1q_1h_2^{-1},q_2h_2^{-1},h_1q_3,q_4h_2^{-1})$$
and the corresponding  tri-moment map is 
$$i^*\circ \sigma(q):=\left(-\frac{1}{4}(|q_1|^4+|q_3|^4)+1,-\frac{1}{4}(|q_1|^4+|q_2|^4+|q_4|^4)+2\right).$$
Thus $Z=(i^*\circ \sigma)^{-1}(0,0)$ is 
is given by
\begin{equation*}
\left\{ \begin{array}{l} |q_1|^4+|q_3|^4=4\\
|q_1|^4+|q_2|^4+|q_4|^4=8\end{array}.\right.
\end{equation*} 
It is not difficult to show that  the orbit space $M_T=Z/\widehat{N}$ is a (non-trivial) $\H\P^1$-bundle on $\H\P^1$ and hence coincides with the blow-up of $\H\P^2$ at a point. 
The action of $\widehat{G}\cong Sp(1)^2$ which  can be defined on $\widehat{N}$-orbits as
 $$g\cdot \widehat{N}q=(h_1q_1h_2^{-1},g_1q_2h_2^{-1},h_1q_3g_2^{-1},q_4h_2^{-1})$$
has clearly  trivial isotropy (and hence  is effective).}
\end{example}
\noindent As proved in e.g. \cite{GGS}, the blow-up of $\HP^2$ at a point, i.e. the manifold $M_T,$ is the  connected sum $\HP^2 \#\ \overline{\HP^2}$ where the symbol $\overline{\HP^2}$ denotes the quaternionic projective space with the reversed orientation.\\
The trapezoid $T$ is obtained via a cut by means of a hyperplane (a straight line) of the standard simplex $\Delta_2$, and the corresponding manifold $M_T$ is the blow-up at a point of $M_{\Delta_2}\cong\HP^2$. The fact that cutting certain polytopes corresponds to blowing-up the associated manifolds is indeed general as shown in the following section.
\begin{remark}\label{nonfunz} {\em Note that if one considers the trapezoid whose vertices are $(0,0);(l+1,0);(0,1);(1,1)$ the action of $N$ cannot be extended to $\widehat{N}\cong Sp(1)^2$ for $l>1$. 
Indeed in this case the kernel $\mathfrak{n}$ of $\pi:\R^4\to\R^2$ is spanned by the vectors $\{(1,0,1,0),(1,l,0,1)\}$ and does not admit a basis $\mathcal {B}$ such that condition (1) in Proposition \ref{Adm} is satisfied.
Therefore we cannot apply the procedure starting from this Delzant polytope.}
\end{remark}

 \section{$4$-plectic reduction and $4$-plectic cut.}
\noindent Starting from a symplectic manifold $(M,\omega)$ acted on by a compact Lie group $K$ in a Hamiltonian fashion with moment map $\mu$, the Marsden-Weinstein reduction is a tool that permits to equip the manifold $\mu^{-1}(x)/K$, when $x\in \mathfrak{k}^*$ is a regular value of $\mu$,  with a symplectic form $\omega_{red}$  (see e.g. \cite{Mar}).\\ In the $4$-plectic setting  the procedure holds under stronger hypotheses  and requires that the starting $4$-plectic form be {\em horizontal}.\\ Let  $(M^{4m},\psi)$  be a $4$-plectic manifold on which the group  $G={Sp(1)}^m$ acts with tri-moment map $\sigma.$ Let $x$ be a regular value in ${\R^m}^*$ for the map $\sigma.$   Consider the $G$-invariant  and smooth $Z_{x}={\sigma}^{-1}{(x)}$. The stabilizers of points in $Z_{x}$ form a group bundle over it, which we assume to be smooth.  We say that these stabilizers form a {\em spheroid bundle} if they are isomorphic to the product of copies of $\Sp(1)$. Then the quotient space $Y_{x}:=Z_{x}/G$, usually known as {\em reduced space} is a smooth manifold. We say that the $4$-plectic form $\psi$ is {\em horizontal} on $Z_{x}$ if and only if the contraction of $\psi$ along every fundamental vector field $\hat\beta$ (for each $\beta\in \mathfrak{g}$) is  zero along  $Z_{x}$. Formally
 $$\iota_{\hat{\beta}}\psi_{|_{Z_{x}}}=0\;\;\;\forall \beta \in \mathfrak{g}$$
 With this notation we recall 
 \begin{theorem}[Theorem 3.1 \cite{F}] \label{reduction} Let  $(M^{4m},\psi)$  be a $4$-plectic manifold on which the group  $G={Sp(1)}^m$ acts with tri-moment map $\sigma$ and   $x$ be a regular value of  $\sigma.$ Assume that the stabilizers of all points in $Z_{x}$ form a smooth spheroid  bundle over $Z_{x}$, and that $\psi$ is horizontal. Then the {\em reduced space} $Y_{x}$ is a smooth manifold admitting a $4$-plectic form $\psi_{red}$, such that $$\pi^*(\psi_{red})=i^*{(\psi)}$$ where $\pi:=Z_{x}\rightarrow Y_{x}$ and $i:=Z_{x}\rightarrow M$ denote respectively the projection and the inclusion map.
 \end{theorem}
 \begin{example}Let $X=\H^2$ with the standard diagonal action of $\Sp(1)$  and let the $4$-form $\psi_h$ be given by
 $$\psi_h=d(|q_1|^4-|q_2|^4)\wedge d(\alpha_1-\alpha_2)\wedge d(\beta_1-\beta_2)\wedge d(\gamma_1-\gamma_2),$$
 where $q_r=x_r+\alpha_r i+\beta_r  j+\gamma_r k$ for $r:=1,2$. The form $\psi_h$ is horizontal. The reduced space is isomorphic  to $\H\P^1$, with an appropriate $4$-plectic form $\psi_{red}.$
 Similarly we can obtain an invariant $4$-plectic form on $\H\P^m$ that we will denote again by $\psi_{red}$ \cite{F}. \end{example}

\noindent We would like to define a $4$-plectic analog of blowing-up $M$ at  a point $p$.  In \cite{GGS} we have proved 
that if $M$ is  a regular  quaternionic manifold of real dimension $4m$, then the blow-up of $M$ at one point is a $4m$-dimensional  regular quaternionic manifold which is  diffeomorphic to
$M\#\H\P^m$.
Note
that topologically $M\#\H\P^m$ can be obtained by removing a ball centered at $p$ and then collapsing the boundary $S^{4m-3}$ along the fibers of the Hopf fibration $S^3\to S^{4m-3}\rightarrow\H\P^m$.  The generalization in the $4$-plectic set up  of the so-called symplectic cutting due to E. Lerman is a particular case of the quaternionic blow-up.
Let $(M^{4m},\psi)$ be a $4$-plectic manifold equipped with a generalized Hamiltonian $Sp(1)^m$-action. Consider the restricted diagonal  $Sp(1)$-action, and let $h:M\rightarrow \R$ be the corresponding tri-moment  map. Let $\varepsilon$ be a regular value of $h$. 
We assume that the $Sp(1)$-action on $h^{-1}(\varepsilon)$ is free. \\We introduce the following notations: we denote by $M_{h>\varepsilon}, M_{h\geq \varepsilon}$ the pre-images of $(\varepsilon,\infty)$ and $ [\varepsilon,\infty)$ under $h : M\rightarrow \R$, and denote by $\overline{M_{h\geq\varepsilon}}$  the manifold which is obtained by collapsing the boundary ${h^{-1}}(\varepsilon)$  of $M_{h\geq \varepsilon}$ along the orbits of the $Sp(1)$-action. The manifold $\overline{M_{h\geq \varepsilon}}$ can be therefore identified with the blow-up of $M$ at a point and
is called {\em the $4$-plectic cut } of $M$. \begin{theorem} \label{TEO7.2}Let $(M^{4m},\psi)$ be a $4$-plectic manifold equipped with a generalized Hamiltonian $Sp(1)^m$-action. Consider the restricted action of a single  $Sp(1)$, and let $h:M\rightarrow \R$ be the corresponding tri-moment  map. Let $\varepsilon$ be a regular value of $h$. 
We assume that the $Sp(1)$-action on $h^{-1}(\varepsilon)$ is free. Whenever the form $\psi\oplus\psi_0$ on $M\times \H$ is horizontal along ${(h-\frac{1}{4}|q|^4)}^{-1}(\varepsilon),$ there is a natural $4$-plectic structure $\Psi_{\varepsilon}$ on $\overline{M_{h\geq\varepsilon}}$ such
that the restriction of $\Psi_{\varepsilon}$ to $M_{h>\varepsilon} \subseteq \overline{M_{h\geq \varepsilon}}$ equals $\psi$ .
\end{theorem}
\begin{proof}
Consider the  product $(M \times\H, \Phi=\psi\oplus \psi_0)$ and the Hamiltonian $Sp(1)$-action $$\lambda (m,q)=(\lambda(m),\lambda q);\;\;\;\lambda\in Sp(1),m\in M, \;q\in \H.$$
The tri-moment map is $$F(m, q) = h(m)- \frac{1}{4} |q|^4.$$
Observe the following identification
$$F^{-1}(\varepsilon) = \{(m, q)\;| h(m) > \varepsilon, |q|^2 =2\sqrt{(h(m) - \varepsilon)}\} \cup \{(m, 0)|h(m) = \varepsilon\}$$$$= {M_{h>\varepsilon} \times  S^3 \cup h^{-1}(\varepsilon)}.$$
So that
 $$F^{-1}(\varepsilon)/Sp(1)=\left({M_{h>\varepsilon} \times S^3 \cup h^{-1}(\varepsilon)}\right)/Sp(1)=\overline{M_{h\geq\varepsilon}}$$  since  a fundamental set in ${M_{h>\varepsilon}} \times S^3$ for the action of $Sp(1)$ is given by ${M_{h>\varepsilon} }\times \{1\}.$ 
The assumption on the horizontality of $\Phi$ implies that we can apply Theorem \ref{reduction}, and equip
 $F^{-1}(\varepsilon)/Sp(1)$ with  a $4$-plectic structure that equals $\psi$ when restricted to the open submanifold $M_{h>\varepsilon}$. 
 \end{proof}
 \begin{definition} Let $M$ be  a $4m$-dimensional manifold obtained as a reduced space from the $4$-plectic manifold $(N,\psi_h)$ acted on in a generalized Hamiltonian fashion by a group $Sp(1)^k$, with tri-moment map $\sigma$, and let $\psi_h$ be horizontal on $\sigma^{-1}(x)$ for some regular value $x$ of $\sigma$.  Then $M$ equipped with the natural $4$-plectic form $\psi^{x}_{red}$ such that $\pi^*(\psi^{x}_{red})=i^*(\psi_h)$, is said to be  {\em  obtained by reduction}.
 \end{definition}
 \noindent Those manifolds  $M$  obtained by reduction  are special:  as an application of Theorem \ref{TEO7.2} we will see that any $4$-plectic cut of such an $M$ can be equipped with a canonical $4$-plectic structure.

 \subsection{$4$-plectic form on the blow-up}

  \label{blow}
 We here do explicit calculations for the case $M=\H\P^2$,  using the notation established in the proof of Theorem \ref{TEO7.2}.\\
 We  want  to prove that the form $\Phi=\psi^{\varepsilon}_{red}\oplus \psi_0$ on $\H\P^2\times \H$ is horizontal on ${F^{-1}(\varepsilon)}$ so that, applying Theorem \ref{TEO7.2}, we can equip the blow-up of $\H \P^2$ with a $4$-plectic form.  More precisely we show that $\iota_{\hat \beta}\Phi |_{F^{-1}(\varepsilon)}=0$, for all $\beta\in \mathfrak{sp}(1)$. \\ First note that the contraction of the form $\Phi$ along $\hat\beta_{q}$ for $q\in S^3$ is zero. Indeed for example,  using the standard basis of $\mathfrak{sp}(1)$, if $\beta=H,$ $$\widehat H_{q}=-x_2\frac{\partial}{\partial x_1}+x_1\frac{\partial}{\partial x_2}-x_4\frac{\partial}{\partial x_3}+x_3\frac{\partial}{\partial x_4}$$
the contraction of $\Phi$ along  it, is therefore given by 
$$\begin{aligned}
&\iota_{\widehat{H}_q}\Phi=\iota_{\widehat{H}_q}\psi_0\\
&=x_2 dx_2\wedge dx_3\wedge dx_4+x_1 dx_1\wedge dx_3\wedge dx_4-x_3 dx_1\wedge dx_2\wedge dx_3-x_4 dx_1\wedge dx_2\wedge dx_4. 
\end{aligned} $$ The claim follows since $$S^3=\{q=x_1+ix_2+jx_3+kx_4\in \H\;|\;|q|^2=|x_1|^2+|x_2|^2+|x_3|^2+|x_4|^2=c\}$$ and thus 
$x_1dx_1=-\sum_{i=2}^4 x_i dx_i$ on it. Analogously 
for $\widehat X_q$ and $\widehat Y_q$ we have on $S^3$, $\iota_{\widehat{X}_q}\Phi=\iota_{\widehat{Y}_q}\Phi=0.$
 In \cite{F}(p.337) it is proven  that $(\H\P^2,\psi^{\varepsilon}_{red})$ can be obtained via reduction from $(\H^3,\psi_h)$, acted on by $Sp(1)$ with tri-moment map  $\s$, where $\psi_h$ is horizontal.\\  Note that $Sp(1)^2$ acts on $\s^{-1}(\varepsilon)$ and on $\H\P^2$. Here we use the fact that the projection $\pi:\s^{-1}(\varepsilon)\subseteq \H^3 \rightarrow \s^{-1}(\varepsilon)/Sp(1)$ is $Sp(1)^2$-equivariant. By the equivariance, for every $m\in \s^{-1}(\varepsilon)$ $$\pi_*(\hat\beta_m)=\hat \beta_{\pi(m)}.$$ 
Take a point $p=\pi(m)$ in $\H\P^2=\s^{-1}(\varepsilon)/Sp(1)$, 
$$\iota_{\hat\beta_p}\Phi_{|_{F^{-1}(\varepsilon)}}=\iota_{\hat\beta_p}{\psi^{\varepsilon}_{red}}_{|_{\H\P^2_{\s\geq\varepsilon}}}=\iota_{\hat\beta_{\pi(m)}}{\psi^{\varepsilon}_{red}}_{|_{\H\P^2_{\s\geq\varepsilon}}}=\iota_{\pi_*(\hat \beta_m)}{\psi^{\varepsilon}_{red}}_{|_{\H\P^2_{\s\geq\varepsilon}}}.$$
The tangent space of $\H\P^2_{\s\geq\varepsilon}$ at $p=\pi(m)$ corresponds via $\pi_*$ to a subspace of  the tangent space of $\s^{-1}(\varepsilon)$ at $m$. Now
$$\iota_{\pi_*(\hat \beta_m)}{\psi^{\varepsilon}_{red}}=\iota_{(\hat \beta_m)}{\pi^*{\psi^{\varepsilon}_{red}}}=\iota_{(\hat \beta_m)}{\psi_{h}}$$
since the contraction of $\psi^{\varepsilon}_{red}$ along ${\pi_*(\hat \beta_m)}$ is given by the pull back $$\pi^*(\psi^{\varepsilon}_{red})=i^*{\psi_h}.$$ The fact that     
$\psi_h$ is appropriately chosen horizontal on $\s^{-1}(\varepsilon)$ implies that $\iota_{(\hat \beta_m)}{\psi_{h}}=0$. 
\noindent We can give the following general definition.
\begin{definition}
A $4$-plectic manifold $M$, obtained by reduction  from $(N,\psi_h)$ with tri-moment map $\sigma:N\to{\R^k}^*$ for the action of $G=Sp(1)^k$, is said to be obtained by an {\em $Sp(1)^m$-equivariant} reduction if the projection $\pi:\sigma^{-1}{(x)}\to \sigma^{-1}{(x)}/G$ is $Sp(1)^m$-equivariant. \end{definition} 
\noindent With the above notations, with the same argument used for $M=\H\P^2$ it is not difficult to prove 
\begin{theorem}\label{reduction}
Let $M$ be a $4$-plectic manifold obtained via a $Sp(1)^m$-equivariant reduction from $(N,\psi_h)$.
Then
the blow-up of $M$ at a point can be endowed with a $4$-plectic form reduced from the horizontal form $\psi_h$. 
\end{theorem}

\noindent 
In \cite{F} it is proven that all quaternionic flag manifolds can be obtained by equivariant reduction. Therefore the previous theorem can be applied to this class of examples, showing that it is possible to equip their blow-up with a $4$-plectic structure.

\subsection{Polytopes vs quaternionic toric manifolds}
Let $P\subseteq {\mathbb{R}^m}^*$ be obtained by cutting the standard simplex $\Delta_m$ with $d-m-2$ hyperplanes parallel to the original facets (in order to have $d-1$ facets),
 and let $M_P$ be the corresponding manifold. We prove here that if we cut another time $P$ with a hyperplane parallel to one of its facets then the manifold that we get is the blow-up at a point of $M_P$.\\
With the notation of Section \ref{poli}, the kernel $N_1=\ker \pi$ of the map $\pi:(S^1)^{d-1}\to ( S^1)^m$ has real dimension $d-m-1$. Let $\widehat{N_1}\cong Sp(1)^{d-m-1}$ be its extension  and let 
  $$\s_{\widehat{N_1}}:\mathbb{H}^{d-1}\rightarrow {\rr^{d-m-1}}^*$$
   be the tri-moment map associated with the action of $\widehat{N}_1$ on $\HH^{d-1}$. 
By construction, the manifold $M_P$ is given by
 \begin{equation}
 M_P=\s_{\widehat{N_1}}^{-1}(a_1,a_2\ldots,a_{d-m-1})/\widehat{N}_1
 \end{equation}
 where $(a_1,a_2\ldots, a_{d-m-1})\in\mathbb{R}^{d-m-1}$ is determined by the polytope $P$.
 Cutting $P$ with a hyperplane parallel to one of its facets to obtain a new polytope $\widetilde P$ with $d$ facets, we get that the kernel of the new projection $\tilde \pi$ is isomorphic to $N=N_1\times N_2$ where $N_2\cong  S^1$.
 Since the action of $N_1$ is trivial on the $d$-th coordinate of $\HH^d$,  
 it is possible to define the action of the extension $\widehat{N}=\widehat{N}_1\times \widehat{N}_2$ where $\widehat{N}_2\cong Sp(1)$. 
If we enumerate the facets of $P$ as $j=1,\ldots,d-1$ and the cut is parallel to the $j$-th facet, then the  tri-moment map $\s_{\widehat{N}}:\mathbb{H}^{d-1}\times \mathbb{H}\rightarrow {\mathbb{R}^{d-m}}^*$ corresponding to the action of $\widehat N$ on $\HH^d$ is given by \[\s_{\widehat{N}}(q_1,\ldots,q_{d-1},q_d)=(\s_{\widehat{N_1}}(q_1,\ldots,q_{d-1}),\langle \s
(q_1,\ldots,q_{d-1}), e_j\rangle -\frac{1}{4}|q_d|^4 )\]
where $\s$ is the tri-moment map associated with the standard action of $Sp(1)^{d-1}$ on $\HH^{d-1}$ and $e_j$ is the $j$-th element of the standard basis of $\rr^{d-1}$.
Therefore, given $(a_1,a_2\ldots,a_{d-m}) \in \mathbb{R}^{d-m}$, we get
 \begin{equation}
\s_{\widehat{N}}^{-1}(a_1,a_2\ldots,a_{d-m})=
\left\{
\begin{aligned}
\s_{\widehat{N_1}}^{-1}(a_1,a_2\ldots,a_{d-m-1})\times \mathbb{H}\\
\langle\s(q),e_j\rangle-\frac{1}{4}|q_d|^4=a_{d-m}
\end{aligned}
\right .
\end{equation}
where $q=(q_1,q_2,\ldots,q_{d-1})$. Following the procedure, the manifold $M_{\widetilde P}$ is obtained as the quotient of $\s_{\widehat{N}}^{-1}(a_1,a_2\ldots, a_{d-m})$ via the action of $\widehat{N}_1\times \widehat{N}_2$. Now $\widehat{N}_1$ acts only on the first $(d-1)$-coordinates.  So the quotient $$M_{\widetilde{P}}=\s_{\widehat{N}}^{-1}(a_1,a_2\ldots, a_{d-m})/\widehat{N}$$ is given by 
\begin{equation}
\begin{array}{cc}
&M_{\widetilde P}=\{(q,q_d) \in M_P\times \H\;\;|\; \langle\s(q),e_j\rangle -\frac{1}{4}|q_d|^4=a_{d-m}\}/\widehat{N}_2.
\end{array}
\end{equation}
If we denote by $h(q)=\langle\s(q),e_j\rangle$, we get that $M_{\widetilde{P}}=\overline{(M_P)}_{h\geq a_{d-m}}$. 
This fact is relevant, indeed it implies that the manifold $M_{\widetilde P}$ is obtained via $4$-plectic cut from  $M_P$. 
Since  $P$ is obtained cutting appropriately the standard simplex, applying iteratively Theorem \ref{reduction} at each cut,
we can  conclude that $M_{\widetilde P}$ admits a $4$-plectic form. 
\begin{theorem} The manifold corresponding to a polytope with $m+k+1$ facets, obtained via cutting the standard simplex $\Delta_m$  in ${\R^m}^*$ by means of $k$-hyperplanes parallel to facets of $\Delta_m$ is the blow-up at $k$ points of $\mathbb{H}\mathbb{P}^m$. 
Moreover it admits a $4$-plectic form. 
\end{theorem}
 \noindent 
 We observe here that this class of manifolds, thanks to Theorem 3.12  in \cite{GGS}, has also the property of being {\em quaternionic regular}.
 \section{Quaternionic toric manifolds} 
 In the symplectic setting, the Delzant Theorem establishes a one-to-one correspondence between symplectic toric manifolds and Delzant polytopes (up to symplectomorphisms). In the  $4$-plectic case in Theorem \ref{convexity} we obtained a sub-convexity  result on the image of the tri-moment map.  For the class of $4$-plectic manifolds, associated with polytopes obtained cutting the standard simplex,  by means of hyperplanes parallel to its facets,   the image of  the tri-moment map turns out to be convex  and  it coincides with the starting polytope. This establishes, for this class of $4$-plectic manifolds, a correspondence completely analogous to the one stated by the  Delzant Theorem in the symplectic case.\\
 We here present some significant examples in the $4$-plectic setting; we explicitly describe the generalized Hamiltonian action  and  the corresponding  tri-moment map.
 \subsection*{Quaternionic projective spaces} We recall that if $(q_1,\ldots,q_{n+1})\in \H^{n+1}\setminus \{0\},$ then $[q_1,\ldots,q_{n+1}]$ denotes the (right) vector line $\{(q_1\lambda,\ldots,q_{n+1}\lambda)\in \H^{n+1}:\lambda\in \H\}$ of $\H^{n+1}$. As usual $\H\P^n$ denotes the set of (right) vector lines in $\H^{n+1}$.
Using the reduced form obtained in \cite{F} p.337  on $\H \P^m$ acted on freely by the group $Sp(1)^m$
 as $$(\lambda_1,\lambda_2,\ldots\lambda_{m})[q_1:q_2:\cdots:q_{m+1}]=[\lambda_1q_1:\lambda_2q_2:\cdots\lambda_mq_m:q_{m+1}]$$  the tri-moment map turns out to be   $$\s([q_1:q_2:\cdots:q_{m+1}])=-\left(\frac{|q_1|^4}{\sum_{i=1}^{m+1}|q_i|^4}, \frac{|q_2|^4}{\sum_{i=1}^{m+1}|q_i|^4},\ldots,\frac{|q_m|^4}{\sum_{i=1}^{m+1}|q_i|^4}\right)\in{\R^{m}}^*.$$  The image, which is an $m$-simplex, is given by the convex envelope of the images of the points fixed by the group $Sp(1)^m$ that coincide, in this case, with the common critical points of the components   of the tri-moment map.
 \subsection*{Blow up of $\H\P^2$} 
In general the idea is the following: we start with the quaternionic projectve space $\HP^2$ acted on by the group $Sp(1)^2$ 
$$(\lambda_1,\lambda_2)[q_1:q_2:q_3]=[\lambda_1q_1:\lambda_2 q_2:q_3].$$
At each blow-up corresponds an extended action.
We blow-up  first at the point ${[0:1:0]},$ then at ${[1:0:0]}$ and finally at ${[0:0:1]}$. Thus the action at the third step (after three blow-ups)
is the following
\begin{eqnarray}\label{act}(\lambda_1,\lambda_2)([q_1:q_2:q_3],[p_1:p_2],[r_1:r_2],[s_1:s_2])=\\=([\lambda_1q_1:\lambda_2 q_2:q_3],[\lambda_1p_1:p_2],[\lambda_2r_1:r_2],[\lambda_1s_1:\lambda_2s_2]).\end{eqnarray}
If ${(\psi_1)}_{red},{(\psi_2)}_{red}$ denote the reduced $4$-plectic structures on $\HP^1$  and $\HP^2$ respectively (see \cite{F} for the precise expression), we can equip the exceptional divisors $E_i\cong\HP^1$ with the $4$-plectic structures $\alpha_i {(\psi_1)}_{red}$ with $i=1,2,3$. Thus the tri-moment map is  given by $({\sigma}_1,\sigma_2)$ where
\begin{eqnarray}{4\sigma}_1=-\frac{|q_1|^4}{|q_1|^4+|q_2|^4+|q_3|^4}-\alpha_1\frac{|p_1|^4}{|p_1|^4+|p_2|^4}-\alpha_3\frac{|s_1|^4}{|s_1|^4+|s_2|^4}\\
{4\sigma}_2=-\frac{|q_2|^4}{|q_1|^4+|q_2|^4+|q_3|^4}-
\alpha_2\frac{|r_1|^4}{|r_1|^4+|r_2|^4}-\alpha_3{\frac{|s_2|^4}{|s_1|^4+|s_2|^4}}.\end{eqnarray}
Note that if $\alpha_i=0$  for $i=1,2,3$ we find again the tri-moment map of the quaternionic projective space, and for $\alpha_2=\alpha_3=0$ we find the tri-moment map of the blow-up at ${[0:1:0]}$ (the first step), and analogously for $\alpha_1=\alpha_3=0$ the blow-up at   ${[1:0:0]}$  and  for $\alpha_1=\alpha_2=0$ at  ${[0:0:1].}$  Moreover all the $\alpha_i$ must be less or equal to $1$. We finally observe that, in order to obtain an image which is a polytope to which one can apply our procedure necessarily $\alpha_1$ must equal  to $\alpha_2$ otherwise the slope of the edge is not a multiple of $\frac{\pi}{4}$. 
Looking at the action in equation (\ref{act}), we can observe that the fixed points are more than $6$.  However the image is given by the convex envelope of all the fixed points, and is an hexagon (admissible for our procedure only if $\alpha_1=\alpha_2$).

\begin{remark}\label{conv} {\em One can observe that in all  the cases considered so far i.e. when  $M$ equals $\H^m$, $\H\P^{m-1}$  and their iterated  blow-ups, the tri-moment map ${\s}$ for the $Sp(1)^m$ action is given by the composition of the usual moment map $\nu$ for the ${(S^1)}^m$ action on the complex manifolds $\C^m$, $\C\P^{m-1}$ and their iterated blow-ups, with the surjective map $\alpha$ defined on $\H^m$  with values in $\C^m$ as 
$$\alpha(x_1+y_1I_1,\ldots,x_m+y_mI_m)=((x_1+y_1i)^2,\ldots,(x_m+y_m i)^2)$$
and on $\H\P^{m-1}$ with values in $\C\P^{m-1}$ as
$$\alpha([x_1+y_1I_1:\ldots:x_m+y_mI_m])=[(x_1+y_1i)^2:\ldots:(x_m+y_m i)^2]$$
where $x_\ell,y_\ell\in \R,$ $y_\ell\geq 0$ and $I_\ell$ is a purely imaginary unit in $\H$ for any $\ell=1,\ldots,m$.
%
Therefore in all  these cases the image $\s(M)=\nu\circ\alpha(M)$ is  the same of  its complex analog and thus it is a Delzant polytope.}
\end{remark}

\begin{remark}{\em {\bf Action of ${(\H^*)}^m.$}
It is easy to show that  the examples considered so far  admit
a $({\H^*})^m$ action with an open dense orbit.  We point out this fact since it is in complete analogy with what happens for the action of ${(\C^*)}^m$ on the corresponding complex manifolds.\\
 In particular, 
 \begin{enumerate}
 \item the action of ${(\H^*)}^m$ on $\H\P^{m}$ is given by $$(a_1,\ldots,a_m)[q_1:\ldots:q_m:q_{m+1}]:= [a_1 q_1:\ldots:a_mq_m:q_{m+1}]$$ and it has an open dense orbit since the generic stabilizer is trivial.
 \item The  group  ${(\H^*)}^2$  acting on the base space $\H^2$
  as $$(a_1,a_2)(q_1,q_2):= (a_1 q_1,a_2 q_2)$$
 has an open dense orbit  and lifts to the blow up $Bl_p(\H\P^2)$, naturally: the group acts taking a direction in the exceptional fiber $\H\P^1$ to another direction in $\H\P^1$.
  Indeed the action of ${(\H^*)}^2$ on $Bl_p(\H\P^2)$ is given by $$(a_1,a_2)([q_1:q_2:q_3],[p_1:p_2]):=([a_1q_1:a_2 q_2:q_3],[a_1p_1:p_2])$$ and the generic orbit is open and dense.
  \item The same argument also applies for $({\mathbb{H}^*})^m$ on $\H\P^m\#k\overline{ \H\P}^m$.
  \end{enumerate}}
  \end{remark}
  \subsection*{The manifold ${G_2}/ {SO(4)}$}
 Going through the list of Quaternionic K\"ahler manifolds, endowed with the Kraines form, the only one admitting a generalized Hamiltonian action of $Sp(1)^n$ with discrete principal isotropy is the $8$-dimensional quotient  $M=\frac{G_2}{SO(4)}$ with $n=2$. Thus it is a toric quaternionic manifold.\\
 We can actually compute the number of fixed points for the action of $Sp(1)^2$ on $\frac{G_2}{SO(4)}$ proving that the fixed point set is given by a single point.
Indeed  the Euler characteristic of $M$ is $3$, this is given by the quotient of the order of the Weyl group of $G_2$, $|W(G_2)|=12$  over the order of the Weyl group of $SO(4)$, $|W(SO(4))|=4$ since the action is polar the Euler characteristic is equal to the number of  points fixed by a maximal abelian subgroup $T$ in $Sp(1)^2$.
 Let $H$ be the normalizer of $K=SO(4)$ in $G_2.$  The order of the fixed point set of $K$ on $M$ equals the order of $\frac{H}{K}$. The quotient $\frac{H}{K}$ has order $1$ or  $3$. If the order is $3$
then $G_2/H$ would have fundamental group  $\Z_3$  (we here use the homotopy sequence and the connectedness of $\SO(4)$) and would be therefore orientable (since it does not admits subgroups of index two), so that its Euler characteristic should be strictly greater than $1$  and therefore equal to $3$. Hence $\frac{H}{K}$ would have  order $1$ which is a contradiction. So the image, via the moment map, is in this case contained in the convex envelope of a  set of points whose cardinality runs from $1$ to $\# M^T=\chi(M)=3$.\\\\
 A further investigation could clarify whether
$(\H^*)^2$ acts on this manifold with an open dense orbit and
 whether the manifold is quaternionic regular (in the sense of \cite{GGS}). Moreover it would be interesting to understand if
 the image of $\frac{G_2}{SO(4)}$ via the tri-moment map is related with the moment map image of its twistor space.\\\\
\noindent {\bf Acknowledgments}  The authors would like to thank  Victor Guillemin  for suggesting the  study of a quaternionic counterpart of toric manifolds and Fiammetta Battaglia, Fabio Podest\`a and Elisa Prato for interesting and useful conversations.

\end{document}